\documentclass[11pt]{article}
\usepackage[T1]{fontenc}
\usepackage{lmodern}
\usepackage{amsmath,amssymb,amsthm,mathtools}
\usepackage{graphicx}
\usepackage{xcolor}
\usepackage{cite}
\usepackage{booktabs}
\usepackage{microtype}
\usepackage[a4paper,margin=30mm]{geometry}
\usepackage[hidelinks]{hyperref}
\hypersetup{pdftitle={Almost-global convergence to twisted states in discrete Huygens rings},
 pdfauthor={Jorge Buescu, Henrique M. Oliveira}}
\usepackage{placeins}

\newtheorem{theorem}{Theorem}[section]
\newtheorem{proposition}[theorem]{Proposition}
\newtheorem{lemma}[theorem]{Lemma}
\newtheorem{corollary}[theorem]{Corollary}
\theoremstyle{remark}
\newtheorem{remark}[theorem]{Remark}

\newcommand{\ones}{\mathbf1}
\newcommand{\T}{\mathbb T}
\newcommand{\Crit}{\operatorname{Crit}}
\newcommand{\indM}{\operatorname{ind}}
\newcommand{\gradg}{\operatorname{grad}_g}

\title{Almost-global convergence to twisted states in discrete Huygens rings}
\author{Jorge Buescu$^{1}$, Henrique M. Oliveira$^{2,*}$}
\date{}

\begin{document}
\maketitle
\vspace{-1.1em}

\begin{center}
\begin{minipage}{0.90\textwidth}
\centering\small
$^{1}$ Departamento de Matem\'atica, Faculdade de Ci\^encias, Universidade de Lisboa,\\
CEMS.UL -- Center for Mathematical Studies, Campo Grande, 1749-016 Lisbon, Portugal\par
\vspace{0.3em}
$^{2}$ Centro de An\'alise Matem\'atica, Geometria e Sistemas Din\^amicos (CAMGSD),\\
Departamento de Matem\'atica, Instituto Superior T\'ecnico, Universidade de Lisboa,\\
Av.~Rovisco Pais 1, 1049-001 Lisbon, Portugal\par
\vspace{0.3em}
$^{*}$ Corresponding author: \href{mailto:henrique.m.oliveira@tecnico.ulisboa.pt}{henrique.m.oliveira@tecnico.ulisboa.pt}
\end{minipage}
\end{center}
\vspace{0.4em}

\begin{abstract}
We study the discrete Huygens phase map on the ring $C_n$, the simplest
oscillator network in which the coupling can be frustrated. Reciprocal
coupling favours phase opposition between neighbours, while the edge phase
differences must sum to zero modulo $2\pi$ around the cycle; for odd $n$
these two requirements are incompatible. The parity of $n$ determines the
absolute minima of the potential. When $4\mid n$, the critical
set also contains circles of critical points, all of which meet at two
completely degenerate configurations.

After quotienting by the common phase, the map is a gradient step for an
intrinsic potential on an $(n-1)$-torus, and the cycle spectrum gives an
explicit interval of interaction strengths on which every orbit converges to
a critical point. Our main result is that, on this whole interval and for
every $n\ge3$, almost every orbit converges to a uniform twisted minimum,
including at interaction strengths for which the map is not invertible. When
$4\mid n$, the non-minimal points of the critical circles are treated by a
centre-stable manifold argument. At the two degenerate intersections, a
local blow-up reduces the problem to the cubic leading term of the
potential and its restriction to the unit sphere. Orbits approaching along
directions where this restriction is positive form a null set, because the
cubic is harmonic in the intrinsic metric. No orbit approaches along its
zero level, where the critical points are non-degenerate saddles on
invariant rays. Hence the degenerate equilibria also have basins of
zero measure.
\end{abstract}

\vspace{0.3em}
\begin{center}
\begin{minipage}{0.90\textwidth}
\end{minipage}
\end{center}

\clearpage
\section{Introduction}
\label{sec:introduction}

Consider $n$ identical phase oscillators on a closed loop, each coupled
reciprocally to its two nearest neighbours, so that the interaction graph
is the cycle $C_n$. The Huygens interaction drives neighbouring
oscillators towards phase opposition \cite{OliveiraMelo2015}, as in the
anti-phase synchronisation of pendulum clocks reported by Huygens in 1665
\cite{Huygens1665}; attractive Kuramoto coupling instead favours local
alignment \cite{WileyStrogatzGirvan2006}. On a ring, this preference
competes with closure: the edge phase differences must sum to zero
modulo $2\pi$. Even cycles admit opposition on every edge, whereas odd
cycles retain a phase mismatch. When $n$ is divisible by four,
the critical set of the potential contains, besides isolated points,
intersecting circles of critical configurations. We study the iterates of the corresponding
discrete Huygens phase map, in which the interaction strength determines
both the multipliers of the fixed points and the invertibility of the
map.

The discrete Huygens map was introduced for two clocks in
\cite{OliveiraMelo2015} and extended to three clocks with all-to-all
and aligned nearest-neighbour coupling in
\cite{DAnielloOliveira2024,BuescuDAnielloOliveira2025Aligned}.
A Lyapunov function for the aligned model was obtained in
\cite{BuescuDAnielloOliveira2025Lyapunov}, and in
\cite{BuescuOliveira2025MorseSmale} we constructed discrete Lyapunov
functions for three-clock synchronisation diffeomorphisms and proved
that they are Morse--Smale, so that phase locking persists under small
perturbations of the natural frequencies. General reciprocal phase
networks admit an intrinsic potential and a metric on the common-phase
quotient \cite{BuescuDAnielloOliveira2026General}. We specialise this
construction to cycles, where the closure constraint determines the
geometry of the critical set.

Stationary configurations of periodic cosine potentials were classified
by Mehta and Kastner \cite{MehtaKastner2011}; equilibrium counts and
non-isolated families on cycles appear in
\cite{LindbergZachariahBostonLesieutre2023,
ChenKorchevskaiaLindberg2026,Sclosa2023}, and signed-Laplacian methods
give related index information \cite{BronskiDeVilleFerguson2016}.
For the continuous-time Kuramoto equations, twisted equilibria, their
stability and winding numbers are analysed in
\cite{WileyStrogatzGirvan2006,GirnykHaslerMaistrenko2012,
DelabaysColettaJacquod2016,Ferguson2018,MedvedevWright2017,
Yagasaki2026}, basin geometry in \cite{ZhangStrogatz2021,Groisman2026},
and non-global minima of synchronisation potentials in
\cite{LingXuBandeira2019}.

Our results concern the discrete map. Let $\Lambda_n$ be the largest
eigenvalue of the cycle Laplacian. For interaction strengths
$0<h<2/\Lambda_n$ the potential decreases strictly along orbits, and
every orbit converges to a critical point. The main result,
Theorem~\ref{thm:almost-global}, states that on this whole interval and
for every $n\ge3$ almost every orbit converges to a uniform twisted
minimum, including for $h\ge1/\Lambda_n$, where the map is no longer
guaranteed to be a diffeomorphism. Here ``almost-global'' means that the
union of the basins of the locally minimising twists has full measure;
for $n\ge6$ there are several such twists, and the initial condition
determines which one is reached. For $4\nmid n$, the exclusion of
non-minimal limits follows from strict-saddle avoidance results
\cite{LeeSimchowitzJordanRecht2016,PanageasPiliouras2017,
SchaefferMcCalla2020,CheriditoJentzenRossmannek2024}; when $4\mid n$,
the same results apply along the critical circles away from their
intersections. The new step concerns the two intersections, where the
derivative of the map is the identity. There a local blow-up, based on
the harmonicity of the cubic leading term of the potential in the
intrinsic metric and on the invariance of the rays tangent to the
critical circles, shows that their basins have zero measure.

Other discrete-time phase-network models are treated in
\cite{VasylenkoMaistrenkoHasler2004,SarletteSepulchreLeonard2006,
ZhangZhu2020,Shim2020,ZhaiXiaoZheng2023,JafarianMamduhiJohansson2024,
KalmarNagyHorvathVincze2025,Wu2026}, and reviews of coupled oscillators
are given in \cite{AcebronEtAl2005,ArenasEtAl2008,DorflerBullo2014}.
Coupled oscillator networks are realised in electrical circuits and
micromechanical arrays
\cite{EnglishZengMertens2015,BhaskarShayakRandZehnder2021}, where the
collective behaviour involves relative phase patterns as well as
frequency locking, and in oscillator-based hardware
\cite{MoyEtAl2022,LoEtAl2023}.

Section~\ref{sec:model} introduces the quotient dynamics and the
finite-step estimates. Section~\ref{sec:critical} classifies the
critical set and its Morse indices. Section~\ref{sec:attractors}
identifies the attracting twists and shows that strict saddles attract
only a null set. Section~\ref{sec:parabolic} treats the degenerate
intersections and proves the main theorem. Section~\ref{sec:examples}
describes the rings $C_3$ to $C_8$, and Section~\ref{sec:discussion}
concludes. Appendix~\ref{app:slow-dichotomy} proves a slow-time
dichotomy lemma used in Section~\ref{sec:parabolic}.

\section{The ring and its quotient dynamics}
\label{sec:model}

\subsection{The phase map and intrinsic potential}

Let $n\ge3$ and consider $n$ identical phase oscillators with phases
$\theta_i\in\T=\mathbb R/(2\pi\mathbb Z)$, coupled along the
unweighted cycle $C_n$; indices are read modulo $n$. For interaction
strength $h>0$, the full map $F_h:\T^{n}\to\T^{n}$ is
\begin{equation}
\label{eq:fullmap}
 (F_h(\theta))_i
 =\theta_i+h\bigl[\sin(\theta_i-\theta_{i-1})
                  +\sin(\theta_i-\theta_{i+1})\bigr]
 \pmod{2\pi},\qquad i=1,\ldots,n.
\end{equation}
Its intrinsic discrete potential is
\begin{equation*}
 \mathcal V(\theta)=\sum_{i=1}^{n}\cos(\theta_i-\theta_{i-1}),
\end{equation*}
and direct differentiation gives
\begin{equation*}
 F_h(\theta)=\theta-h\nabla\mathcal V(\theta)\pmod{2\pi}.
\end{equation*}
The common phase advance has been removed by working in a co-rotating
frame.

\subsection{Edge differences and the reduced metric}
\label{subsec:quotient}

Write $\ones=(1,\ldots,1)^T\in\mathbb R^{n}$, let $e_i$ be the
standard basis vectors, and let $B$ be the incidence matrix with
columns $b_i=e_i-e_{i-1}$. The edge differences are
\begin{equation*}
 z=B^T\theta,\qquad z_i=\theta_i-\theta_{i-1}\pmod{2\pi}.
\end{equation*}
The orientation only fixes signs; the coupling is undirected. The edge
differences satisfy $\ones^Tz\equiv0\pmod{2\pi}$, so they lie in the
$(n-1)$-dimensional torus
\begin{equation*}
 \mathcal Z=\{z\in\T^{n}:\ones^Tz\equiv0\pmod{2\pi}\}.
\end{equation*}
Two configurations have the same edge differences exactly when they
differ by a common phase, and every $z\in\mathcal Z$ arises in this way:
once one phase is fixed, the others are recovered successively, and
closure makes the last edge consistent. Thus $\mathcal Z$ is the
quotient of $\T^n$ by common-phase shifts, described by $n$ edge
coordinates subject to one constraint. We denote by $\mu$ the
normalised Lebesgue measure on $\mathcal Z$.

A \emph{uniform $q$-twisted state}, or \emph{$q$-twist}, has the same
difference on every edge:
\begin{equation}
\label{eq:uniformtwist}
 z=\frac{2\pi q}{n}\ones,\qquad q=0,\ldots,n-1.
\end{equation}
Negating the phases exchanges $q$ and $n-q$ modulo $n$.  If $n$
is even, $q=n/2$ is the state with phase opposition on every edge.

Define the edge potential and its restriction by
\begin{equation*}
 \Psi(z)=\sum_{i=1}^{n}\cos z_i,\qquad V=\Psi|_{\mathcal Z}.
\end{equation*}
Then $\mathcal V(\theta)=\Psi(B^T\theta)$. For this cyclic ordering,
the edge Laplacian $K=B^TB$ has the same entries as the vertex Laplacian
$BB^T$: $2$ on the diagonal and $-1$ at cyclic neighbours. Since
$F_h(\theta)=\theta+hB\sin(B^T\theta)$, the reduced map is
\begin{equation*}
 \bar F_h(z)=z+hK\sin z\pmod{2\pi},\qquad
 B^T\circ F_h=\bar F_h\circ B^T,
\end{equation*}
where sine acts componentwise; it preserves closure because
$\ones^TK=0$. In coordinates,
\begin{equation*}
 (\bar F_h(z))_i=z_i+h(2\sin z_i-\sin z_{i-1}-\sin z_{i+1})
 \pmod{2\pi}.
\end{equation*}

The tangent space of $\mathcal Z$ is $E=\ones^\perp$, with orthogonal
projection $P_E=I-\ones\ones^T/n$, and $K_0=K|_E$ is positive definite.
The flat quotient metric on $E$ is
\begin{equation}
\label{eq:metric}
 g(u,v)=u^TK_0^{-1}v,\qquad u,v\in E,
\end{equation}
with norm $\|u\|_g^2=g(u,u)$. It is induced by the Euclidean norm of
vertex displacements: the zero-mean vertex displacement producing the
edge displacement $u\in E$ is $(B^T|_E)^{-1}u$, and
$\|(B^T|_E)^{-1}u\|^2=u^TK_0^{-1}u$. The $g$-gradient is
\begin{equation*}
 \gradg V(z)=-K\sin z,
\end{equation*}
so the reduced map has the same gradient form,
\begin{equation}
\label{eq:gradientreduced}
 \bar F_h(z)=z-h\gradg V(z)\pmod{2\pi}.
\end{equation}

\subsection{Spectrum and finite-step estimates}
\label{subsec:spectrum}

The eigenvalues of $K$ are
\begin{equation}
\label{eq:spectrum}
 \lambda_j=2-2\cos\frac{2\pi j}{n},\qquad j=0,\ldots,n-1,
\end{equation}
with $\lambda_0=0$ corresponding to $\ones$.  Hence
\begin{equation*}
 \Lambda_{n}:=\lambda_{\max}(K)=
 \begin{cases}
 2+2\cos\dfrac{\pi}{n},&n\text{ odd},\\[3pt]
 4,&n\text{ even}.
 \end{cases}
\end{equation*}
Write $D(z)=\operatorname{diag}(\cos z_1,\ldots,\cos z_{n})$.
The ambient and reduced Hessians are
\begin{equation*}
 D^2\Psi(z)=-D(z),\qquad H_z=-P_ED(z)|_E.
\end{equation*}
Both have Euclidean operator norm at most one.  We write $\Crit(V)$ for
the critical set.  At a non-degenerate critical point, the Morse index
$\indM$ is the number of negative eigenvalues of $H_z$.

Since $\ker K=\operatorname{span}\{\ones\}$, the critical points satisfy
\begin{equation}
\label{eq:critcondition}
 K\sin z=0
 \quad\Longleftrightarrow\quad
 \sin z_1=\cdots=\sin z_{n}.
\end{equation}

We call
\begin{equation}
\label{eq:lyaprange}
 0<h<\frac{2}{\Lambda_{n}}
\end{equation}
the Lyapunov interval, on which
$c_h:=h\left(1-h\Lambda_{n}/2\right)>0$, and
\begin{equation}
\label{eq:diffeorange}
 0<h<\frac1{\Lambda_{n}}
\end{equation}
the invertibility interval.

\begin{theorem}[Descent and invertibility]
\label{thm:descent-invertibility}
Under \eqref{eq:lyaprange},
\begin{equation}
\label{eq:lyapineq}
 V(\bar F_h(z))\le V(z)-c_h(\sin z)^TK\sin z,
\end{equation}
so $V(\bar F_h(z))<V(z)$ for $z\notin\Crit(V)$. Under
\eqref{eq:diffeorange}, $\bar F_h$ is a diffeomorphism of $\mathcal Z$.
\end{theorem}

\begin{proof}
Taylor's inequality for $\Psi$, whose Hessian has norm at most one,
applied to the increment $hK\sin z$, gives
\[
 V(\bar F_h(z))\le V(z)-h(\sin z)^TK\sin z
             +\frac{h^2}{2}(\sin z)^TK^2\sin z.
\]
Since $K^2\le\Lambda_{n}K$ as quadratic forms, this proves
\eqref{eq:lyapineq}. Since $c_h>0$, the correction term vanishes exactly
when $K\sin z=0$, that is, on $\Crit(V)$ by \eqref{eq:critcondition}.

For \eqref{eq:diffeorange}, the derivative on $E$ is
\[
 D\bar F_h(z)|_E=I+hKD(z)|_E.
\]
Since $\|KD(z)|_E\|_2\le\Lambda_{n}$, it is invertible when
$h\Lambda_{n}<1$. Hence every map of the homotopy
$z\mapsto z+thK\sin z$, $0\le t\le1$, is a local diffeomorphism, and
$\bar F_h$ has degree one.  A local diffeomorphism of the compact
connected torus $\mathcal Z$ is a covering map; degree one makes it
one-sheeted and hence a global diffeomorphism.
\end{proof}

\begin{proposition}[Fixed points and convergence]
\label{prop:global}
Under \eqref{eq:lyaprange}, $\operatorname{Fix}(\bar F_h)=\Crit(V)$,
every periodic point is fixed, and every forward orbit converges to a
single critical point.
\end{proposition}

\begin{proof}
Critical points are fixed by \eqref{eq:gradientreduced}.  Conversely, a
fixed point leaves $V$ unchanged, so Theorem~\ref{thm:descent-invertibility}
forces criticality.  The same strict decrease excludes every non-fixed
periodic orbit.

Write $z_m=\bar F_h^{\,m}(z_0)$.  The descent estimate implies
$(\sin z_m)^TK\sin z_m\to0$.  Hence
$\operatorname{dist}(z_m,\Crit(V))\to0$, every accumulation point is
critical, and all accumulation points have the common potential value
$V_\infty=\lim_mV(z_m)$.  In the reduced metric,
\begin{equation*}
 \|\gradg V(z)\|_g^2=(\sin z)^TK\sin z,
\end{equation*}
and the length of the update is
\begin{equation*}
 \delta_m=\|hK\sin z_m\|_g
         =h\|\gradg V(z_m)\|_g.
\end{equation*}
The distance between successive points of $\mathcal Z$ is at most
$\delta_m$.

Let $\omega$ be the set of accumulation points of $(z_m)$. It is
nonempty and compact, $\operatorname{dist}(z_m,\omega)\to0$, and by
the above $\omega\subset\Crit(V)\cap V^{-1}(V_\infty)$. Near each
$p\in\omega$, analyticity gives the \L{}ojasiewicz gradient inequality
\begin{equation*}
 |V(z)-V_\infty|^{1-\vartheta_p}
 \le C_p\|\gradg V(z)\|_g,
 \qquad C_p>0,\quad 0<\vartheta_p\le\tfrac12,
\end{equation*}
on a neighbourhood $U_p$ of $p$; see \cite{AbsilMahonyAndrews2005}.
By compactness, finitely many of these neighbourhoods,
$U_{p_1},\dots,U_{p_N}$, cover $\omega$; their union $U$ is a
neighbourhood of $\omega$, so $z_m\in U$ for all large $m$. With
$\vartheta=\min_i\vartheta_{p_i}$ and $C=\max_iC_{p_i}$, the inequality
\begin{equation*}
 |V(z)-V_\infty|^{1-\vartheta}\le C\|\gradg V(z)\|_g
\end{equation*}
holds on $U\cap\{|V-V_\infty|\le1\}$, since lowering the exponent can
only decrease the left-hand side there. It therefore holds along the
tail of the orbit. Put $e_m=V(z_m)-V_\infty$.  If some
$e_m=0$, strict descent makes the orbit stationary.  Otherwise, for all
large $m$, $e_m<1$ and concavity of $t^\vartheta$ gives
\[
 e_m^\vartheta-e_{m+1}^\vartheta
 \ge\vartheta e_m^{\vartheta-1}(e_m-e_{m+1})
 \ge\frac{\vartheta(1-h\Lambda_{n}/2)}{C}\,\delta_m.
\]
Thus $\sum_m\delta_m<\infty$.  The orbit is Cauchy and therefore
converges to one critical point.
\end{proof}

\begin{proposition}[Hyperbolicity and Morse index]
\label{prop:hyperbolicity}
Under \eqref{eq:lyaprange}, every non-degenerate critical point $p$ of
$V$ is a hyperbolic fixed point of $\bar F_h$ with $\indM(p)$
multipliers outside the unit circle.
\end{proposition}

\begin{proof}
At a critical point $p$,
\begin{equation}
\label{eq:linearisation}
 D\bar F_h(p)|_E=I-hK_0H_p.
\end{equation}
The matrix $K_0H_p$ is similar to the symmetric matrix
$K_0^{1/2}H_pK_0^{1/2}$ and has the same inertia as $H_p$.  Its
eigenvalues $\eta$ are real and satisfy $|\eta|\le\Lambda_{n}$.  Under
\eqref{eq:lyaprange}, positive $\eta$ give $|1-h\eta|<1$, whereas negative
$\eta$ give $1-h\eta>1$.  Non-degeneracy excludes $\eta=0$.
\end{proof}

Let \(\widetilde F_h:\mathbb R^n\to\mathbb R^n\) denote the lift of
\eqref{eq:fullmap} obtained by omitting reduction modulo \(2\pi\).  Since
\(\ones^TB=0\), it preserves the phase sum:
\begin{equation}
\label{eq:phasesum}
 \ones^T\widetilde F_h(\theta)=\ones^T\theta.
\end{equation}

\begin{corollary}[Convergence of the full phase map]
\label{cor:full-convergence}
Under \eqref{eq:lyaprange}, every orbit of \(F_h\) converges to a
fixed point.
\end{corollary}

\begin{proof}
By Proposition~\ref{prop:global}, the edge differences converge to a
critical configuration.  Every accumulation point of the full orbit
therefore lies on the corresponding common-phase fibre.  The quantity
\(\exp(i\ones^T\theta)\) is constant by \eqref{eq:phasesum}; its level set
meets that fibre in only finitely many points.  Along a lifted orbit $x_{m+1}=\widetilde F_h(x_m)$,
$\widetilde F_h(x_m)-x_m=hB\sin(B^Tx_m)\to0$, since
$\sin z\in\operatorname{span}\{\ones\}$ at each reduced limit
and $B\ones=0$.
A sequence on a compact metric space with a finite accumulation set and
successive distances tending to zero converges to one of its accumulation
points.
\end{proof}

Define
\begin{equation*}
 J_h(z):=\det\bigl(D\bar F_h(z)|_E\bigr).
\end{equation*}

\begin{proposition}[Null preimages]
\label{prop:null-preimages}
For every \(h>0\), \(\bar F_h^{-1}(N)\) is null whenever
\(N\subset\mathcal Z\) is null.
\end{proposition}

\begin{proof}
The function \(J_h\) is analytic on the connected torus \(\mathcal Z\)
and is not identically zero, since
\(J_h(0)=\det(I+hK_0)>0\) (compare
\cite[Lemma~2.4]{CheriditoJentzenRossmannek2024}).  Hence its zero set \(\Sigma_h\) has measure
zero.  The complement \(\mathcal Z\setminus\Sigma_h\) has a countable
cover by inverse-function neighbourhoods on which \(\bar F_h\) is a
local diffeomorphism.  Local inverse maps preserve null sets, while the
remaining part of every inverse image lies in \(\Sigma_h\).
\end{proof}

\section{Critical configurations and their indices}
\label{sec:critical}

By the equal-sine condition \eqref{eq:critcondition}, every critical
configuration takes at most two edge differences:
\begin{equation}
\label{eq:twobranch}
 \alpha\qquad\text{and}\qquad\pi-\alpha\pmod{2\pi}.
\end{equation}
This two-branch description and the finite stationary counts occur in the
periodic XY and Kuramoto literature
\cite{MehtaKastner2011,LindbergZachariahBostonLesieutre2023}; see
\cite{BerglundMedvedevSimpson2025} for the opposite sign of the
potential. We compute the indices directly from the restricted Hessian;
compare the signed-Laplacian approach of \cite{BronskiDeVilleFerguson2016}.

\subsection{Unequal branch multiplicities}
\label{subsec:finitecritical}

Let \(A\subset\{1,\ldots,n\}\), write \(k=|A|\), and let \(\ones_A\)
be its indicator vector. Assign the value \(\pi-\alpha\) to the
coordinates in \(A\) and \(\alpha\) to the remaining coordinates:
\begin{equation}
\label{eq:critconfig}
 z=\alpha\ones+(\pi-2\alpha)\ones_A\pmod{2\pi}.
\end{equation}
When the two multiplicities differ, choose \(A\) to index the less frequent
value. Thus \(\alpha\) is the majority value and
\begin{equation}
\label{eq:dk}
 0\le k\le\left\lfloor\frac{n-1}{2}\right\rfloor,
 \qquad d_k=n-2k>0.
\end{equation}
Closure is equivalent to
\begin{equation}
\label{eq:closurealpha}
 d_k\alpha+k\pi\equiv0\pmod{2\pi}.
\end{equation}
Its \(d_k\) distinct solutions are
\begin{equation}
\label{eq:alphakr}
 \alpha_{k,r}=\frac{2\pi r-k\pi}{d_k}\pmod{2\pi},
 \qquad r=0,\ldots,d_k-1.
\end{equation}
If the multiplicities are equal, then \(k=n/2\), and closure instead
requires
\begin{equation}
\label{eq:halfcondition}
 \frac{n}{2}\pi\equiv0\pmod{2\pi}.
\end{equation}
This is possible exactly when \(n\) is divisible by four. The same
divisibility condition is necessary and sufficient for the two values in
\eqref{eq:twobranch} to coincide in an admissible configuration: all
coordinates then equal \(\pi/2\) or all equal \(3\pi/2\). We write
\(c_{k,r}:=\cos\alpha_{k,r}\).

\begin{theorem}[Critical-point classification]
\label{thm:classification}
If $n\not\equiv0\pmod4$, every critical point of $V$ is represented
exactly once by \eqref{eq:critconfig}--\eqref{eq:alphakr}, is
non-degenerate, and satisfies
\begin{equation}
\label{eq:index-general}
 \indM(z)=
 \begin{cases}
 n-1-k,&c_{k,r}>0,\\
 k,&c_{k,r}<0,
 \end{cases}
 \qquad
 V(z)=d_kc_{k,r}.
\end{equation}
\end{theorem}

\begin{proof}
The observations preceding the theorem exclude equal multiplicities and
coincident branch values. Every critical point therefore has a unique
majority value $\alpha$, a unique exceptional set $A$, and
$\cos\alpha\ne0$. Equation~\eqref{eq:closurealpha} gives exactly the
angles in \eqref{eq:alphakr}. Conversely, each prescribed configuration
satisfies closure and has equal sines, so it is critical. Uniqueness of
$\alpha$ and $A$ shows that the parametrisation has no repetitions,
including for $k=0$.

Put $c=\cos\alpha$, so that $V(z)=(n-k)c-kc=d_kc$. The Hessian form on
$E$ is
\begin{equation*}
 Q_z(u):=u^TH_zu
 =-\sum_{i=1}^{n}\cos z_i\,u_i^2
 =c\sum_{i\in A}u_i^2-c\sum_{i\notin A}u_i^2.
\end{equation*}
For $k\ge1$, decompose $E$ into the zero-sum subspace supported on
$A$, the zero-sum subspace supported on its complement, and the line
generated by
\[
 w=n\ones_A-k\ones.
\]
These subspaces have dimensions $k-1$, $n-1-k$, and one, respectively,
and are mutually orthogonal for $Q_z$. The form has sign $c$ on the
first, sign $-c$ on the second, and
\[
 Q_z(w)=k(n-k)d_kc
\]
on the mixed direction. Since $d_k>0$ and $c\ne0$, this proves
non-degeneracy and \eqref{eq:index-general}. For $k=0$, the restricted
Hessian is $-cI$.
\end{proof}

\begin{corollary}[Critical-point count]
\label{cor:critical-count}
If $n\not\equiv0\pmod4$, then $V$ is Morse with
$n\binom{n-1}{\lfloor (n-1)/2\rfloor}$ critical points.
\end{corollary}

\begin{proof}
There are $\binom{n}{k}$ choices of $A$ and $d_k$ admissible angles
for each choice. Since
\[
 d_k\binom{n}{k}
 =n\left[\binom{n-1}{k}-\binom{n-1}{k-1}\right],
 \qquad \binom{n-1}{-1}=0,
\]
the sum telescopes:
\[
 \sum_{k=0}^{\lfloor (n-1)/2\rfloor}d_k\binom{n}{k}
 =n\binom{n-1}{\lfloor (n-1)/2\rfloor}.
\]
Non-degeneracy is part of Theorem~\ref{thm:classification}.
\end{proof}

\begin{remark}[Index consistency]
\label{rem:morseconsistency}
For $n\not\equiv0\pmod4$, the indices satisfy, as they must, the Morse
relation
\[
 \sum_{z\in\Crit(V)}(-1)^{\indM(z)}
 =\chi(\mathcal Z)=\chi(\T^{n-1})=0,
\]
where \(\chi\) is the Euler characteristic. The count in
Corollary~\ref{cor:critical-count} equals the cycle bound for isolated
complex solutions of the algebraic Kuramoto equations in
\cite{ChenDavisMehta2018}.
\end{remark}

\subsection{Critical circles when the ring size is divisible by four}
\label{subsec:4l}

Throughout this subsection, \(n\equiv0\pmod4\). For \(|A|=n/2\),
closure \eqref{eq:halfcondition} is automatic.
The corresponding configurations form the circle
\begin{equation}
\label{eq:criticalcircle}
 \mathcal C_A
 =\left\{\alpha\ones+(\pi-2\alpha)\ones_A:
          \alpha\in\T\right\}\subset\mathcal Z.
\end{equation}
This equal-partition construction appears in
\cite[Lemma~6]{LindbergZachariahBostonLesieutre2023}; see also
\cite[Section~5.1]{ChenKorchevskaiaLindberg2026}.
Complementary sets define the same circle:
\(\mathcal C_A=\mathcal C_{A^c}\). All these circles pass through
\begin{equation*}
 p_+=\frac\pi2\ones,\qquad p_-=\frac{3\pi}{2}\ones.
\end{equation*}

\begin{theorem}[Critical set for multiples of four]
\label{thm:4l}
If $n\equiv0\pmod4$, then $\Crit(V)$ is the union of the
circles \eqref{eq:criticalcircle} and the configurations
\eqref{eq:critconfig}--\eqref{eq:alphakr} with $c_{k,r}\ne0$. These
configurations are non-degenerate, hence isolated, and satisfy
\eqref{eq:index-general}. Distinct circles meet exactly at $p_+$ and
$p_-$, and $V$ vanishes on every circle.
\end{theorem}

\begin{proof}
For unequal multiplicities, closure gives \eqref{eq:alphakr}. The angles
$\pi/2$ and $3\pi/2$ solve it for every $k$, but their branch values
coincide and represent only $p_+$ and $p_-$. Every other angle gives a
uniquely represented critical point. Its non-degeneracy and index follow
from the Hessian calculation in the proof of
Theorem~\ref{thm:classification}, which uses only $d_k>0$ and
$\cos\alpha\ne0$.

Equal multiplicities give exactly \eqref{eq:criticalcircle}. Its
parametrisation is injective with nonzero derivative, hence an embedding.
Away from $p_\pm$, the two coordinate values determine the partition up
to interchange. Thus only complementary sets define the same circle, and
distinct circles have no other intersections. Finally, $V=d_{n/2}c=0$
on every circle.
\end{proof}

\begin{corollary}[Critical-set counts]
\label{cor:critical-counts-4l}
If $n\equiv0\pmod4$, there are
\begin{equation}
\label{eq:numberisolated4l}
 \frac{n+2}{2}\binom{n}{n/2}-2^{n},
\end{equation}
isolated critical points and
\begin{equation}
\label{eq:numbercircles}
 \frac12\binom{n}{n/2}
\end{equation}
critical circles.
\end{corollary}

\begin{proof}
Removing the two coincident-branch angles leaves
\[
 \sum_{k=0}^{(n-2)/2}(d_k-2)\binom{n}{k}
\]
isolated points. The telescoping identity used in
Corollary~\ref{cor:critical-count}, together with
\[
 \sum_{k=0}^{(n-2)/2}\binom{n}{k}
 =2^{n-1}-\frac12\binom{n}{n/2},
\]
gives \eqref{eq:numberisolated4l}. Complementary subsets define the same
critical circle and no others do, which gives \eqref{eq:numbercircles}.
\end{proof}

\begin{proposition}[Hessian geometry of the critical circles]
\label{prop:circle-hessian}
For $z\in\mathcal C_A\setminus\{p_+,p_-\}$, $H_z$ has $(n-2)/2$
positive and $(n-2)/2$ negative eigenvalues, and kernel
$T_z\mathcal C_A$. Moreover, $H_{p_\pm}=0$ and
$D\bar F_h(p_\pm)|_E=I$.
\end{proposition}

\begin{proof}
Put $c=\cos\alpha\ne0$.  The zero-sum subspaces supported on $A$ and
$A^c$ each have dimension $(n-2)/2$, and the Hessian form is definite on
them with opposite signs.  The remaining direction is the tangent vector
\[
 w=\ones-2\ones_A.
\]
Writing $D=\operatorname{diag}(\cos z_i)$ gives
\[
 H_zw=-P_EDw=P_E(-c\ones)=0.
\]
These three subspaces exhaust $E$.  At $p_\pm$ every cosine vanishes, so
$H_{p_\pm}=0$; equation~\eqref{eq:linearisation} then gives the identity
derivative.
\end{proof}

Equilibria at which every edge cosine vanishes, such as $p_\pm$, are
the completely degenerate equilibria studied for continuous Kuramoto
networks in \cite{Sclosa2024}.

\begin{remark}[Local Morse--Bott geometry]
Each critical circle is Morse--Bott away from $p_\pm$; the union of
circles is not a submanifold of $\mathcal Z$.
\end{remark}

\begin{remark}[The intersections are not extrema]
\label{rem:notextrema}
Near \(p_+\), write \(z=p_++y\) with \(y\in E\). Taylor expansion gives
\begin{equation*}
 V(p_++y)=\frac16\sum_{i=1}^{n}y_i^3+O(\|y\|^5).
\end{equation*}
The linear term vanishes because \(\ones^Ty=0\). Taking
\(y=(2t,-t,-t,0,\ldots,0)\) gives \(V=t^3+O(t^5)\), with both signs
arbitrarily close to \(p_+\). At \(p_-\) the expansion has the opposite
sign. Hence neither intersection is a local extremum.
\end{remark}

\begin{proposition}[Multipliers along the critical circles]
\label{prop:circlemultipliers}
For \(z\in\mathcal C_A\setminus\{p_+,p_-\}\) and \(0<h<1/2\),
\(D\bar F_h(z)|_E\) has the multiplier one along \(T_z\mathcal C_A\),
\((n-2)/2\) multipliers of modulus less than one and \((n-2)/2\)
greater than one.
\end{proposition}

\begin{proof}
The operator \(K_0H_z\) is similar to the symmetric matrix
\(K_0^{1/2}H_zK_0^{1/2}\), whose inertia is given by
Proposition~\ref{prop:circle-hessian}. Its eigenvalues \(\eta\) satisfy \(|\eta|\le4\),
since \(\|K_0\|_2=4\) and \(\|H_z\|_2\le1\).
By \eqref{eq:linearisation}, the corresponding multipliers are
\(1-h\eta\). A positive \(\eta\) gives modulus less than one, a
negative \(\eta\) gives a multiplier greater than one, and the kernel
gives the tangent multiplier one.
\end{proof}

\begin{proposition}[Instability of the intersection points]
\label{prop:degenerateunstable}
For \(0<h<1/2\), neither \(p_+\) nor \(p_-\) is Lyapunov stable.
\end{proposition}

\begin{proof}
Consider \(p_+\). By Theorem~\ref{thm:4l}, a sufficiently small closed
neighbourhood \(U\) contains no isolated critical point. Every critical
point in \(U\) therefore lies on a critical circle and has potential
zero. If \(p_+\) were Lyapunov stable, an initial condition sufficiently
close to it would have its entire forward orbit in \(U\).
Remark~\ref{rem:notextrema} allows such an initial condition \(z_0\)
with \(V(z_0)<0\). Proposition~\ref{prop:global} makes its orbit
converge to a critical point \(p\in U\), whereas descent implies
\[
 V(p)\le V(z_0)<0.
\]
This contradicts the zero potential of the critical set in \(U\).
The argument at \(p_-\) is identical.
\end{proof}

\section{Uniform twists and attracting states}
\label{sec:attractors}

A \emph{sink} is a hyperbolic fixed point whose multipliers all have
modulus less than one. For the uniform $q$-twist \eqref{eq:uniformtwist}, set
$c_q:=\cos(2\pi q/n)$.

\begin{proposition}[Multipliers of a uniform twist]
\label{prop:twistmultipliers}
At the uniform $q$-twist, the multipliers of $D\bar F_h|_E$ are
$1+hc_q\lambda_j$, $j=1,\ldots,n-1$.
\end{proposition}

\begin{proof}
Here $D(z)=c_qI$, so $D\bar F_h(z)|_E=I+hc_qK_0$; now use
\eqref{eq:spectrum}.
\end{proof}

\begin{theorem}[Local minima and sinks]
\label{thm:minima}
The local minima of $V$ are the uniform twists with $c_q<0$, and such a
twist is a sink if and only if
\begin{equation}
\label{eq:exactsinkrange}
 0<h<\frac{2}{-c_q\Lambda_{n}}
\end{equation}
holds. Under \eqref{eq:lyaprange}, these twists are exactly the sinks of
$\bar F_h$.
\end{theorem}

\begin{proof}
For isolated critical points, Theorems~\ref{thm:classification}
and~\ref{thm:4l} show that index zero requires $k=0$ and
$\cos\alpha<0$, that is, a uniform twist with $c_q<0$.
When $n$ is divisible by four, Proposition~\ref{prop:circle-hessian}
gives a negative transverse Hessian direction at every point of a
critical circle away from $p_\pm$, while Remark~\ref{rem:notextrema}
excludes $p_\pm$. Thus there are no other local minima.

By Proposition~\ref{prop:twistmultipliers}, a twist with $c_q<0$ is a
sink exactly when $-hc_q\lambda_j<2$ for every $j$, that is, when
\eqref{eq:exactsinkrange} holds. Every isolated non-minimal critical point has an unstable multiplier by
Proposition~\ref{prop:hyperbolicity}. Points on the critical circles away
from $p_\pm$ have unstable multipliers by
Proposition~\ref{prop:circlemultipliers}, and $p_\pm$ are not Lyapunov
stable by Proposition~\ref{prop:degenerateunstable}. Hence there are no
other sinks. Since $|c_q|\le1$, every local minimum satisfies
\eqref{eq:exactsinkrange} throughout \eqref{eq:lyaprange}.
\end{proof}

\begin{corollary}[Number of sinks]
\label{cor:number-sinks}
Under \eqref{eq:lyaprange}, the number of sinks is
\begin{equation}
\label{eq:numsinks}
 \begin{cases}
  2\left\lfloor\dfrac{n+1}{4}\right\rfloor,
       &n\text{ odd},\\[1.2ex]
  \dfrac{n}{2},&n\equiv2\pmod4,\\[1.2ex]
  \dfrac{n-2}{2},&n\equiv0\pmod4.
 \end{cases}
\end{equation}
\end{corollary}

\begin{proof}
By Theorem~\ref{thm:minima}, the sinks are the uniform twists with
$n/4<q<3n/4$; count these integers.
\end{proof}

\begin{corollary}[Absolute minima]
\label{cor:absolute-minima}
If $n$ is even, $V$ has the unique absolute minimum
\begin{equation}
\label{eq:antiphase}
 z=\pi\ones;
\end{equation}
if $n$ is odd, its absolute minima are
\begin{equation}
\label{eq:absolute-minima}
 z=\left(\pi-\frac{\pi}{n}\right)\ones,
 \qquad
 z=\left(\pi+\frac{\pi}{n}\right)\ones.
\end{equation}
The minimum value is
\begin{equation}
\label{eq:absolutemin}
 \min_{z\in\mathcal Z}V(z)=
 \begin{cases}
  -n,&n\text{ even},\\[.5ex]
  -n\cos\dfrac{\pi}{n},&n\text{ odd}.
 \end{cases}
\end{equation}
\end{corollary}

\begin{proof}
By Theorem~\ref{thm:minima}, it suffices to minimise $V=nc_q$ over the
twists. If $n$ is even, $q=n/2$ is the unique value with $c_q=-1$. If
$n$ is odd, the angles $2\pi q/n$ closest to $\pi$ are $\pi\pm\pi/n$,
and only these minimise the cosine.
\end{proof}

On odd rings, the two absolute minima spread the phase mismatch
uniformly over the edges. Local minima other than the absolute ones
first occur in $C_6$ among even rings and in $C_7$ among odd rings.

For even \(n\), the translation \(z\mapsto z+\pi\ones\) preserves
\(\mathcal Z\) and conjugates \(\bar F_h\) to
\(z\mapsto z-hK\sin z\), the Euler step of the attractive Kuramoto
ring. Under this conjugacy, the condition \(c_q<0\) becomes the
classical \(n/4\) rule for stable twisted states
\cite{WileyStrogatzGirvan2006}; for the discrete map, stability also
requires \eqref{eq:exactsinkrange}.

On even rings both spectral thresholds are sharp: at phase opposition,
\(D\bar F_h|_E=I-hK_0\) is singular at \(h=1/4\) and has the
multiplier \(-1\) at \(h=1/2\). On odd rings, the absolute minima have
\(c_q=-\cos(\pi/n)\) and remain sinks up to
\begin{equation}
\label{eq:odd-stability-gap}
 h=\frac{2}{\cos(\pi/n)\Lambda_n}>\frac{2}{\Lambda_n}.
\end{equation}

Let \(\mathcal S\) be the set of critical points at which the Hessian has
a negative eigenvalue. By Theorem~\ref{thm:minima} and
Proposition~\ref{prop:circle-hessian}, \(\mathcal S\) consists of all
critical points except the local minima and the points \(p_\pm\) (when
\(4\mid n\)). The basin of a fixed point
\(p\) is
\(\mathcal B(p)=\{z\in\mathcal Z:\bar F_h^{\,k}(z)\to p\}\), and
\(\mathcal B(\mathcal S)=\bigcup_{p\in\mathcal S}\mathcal B(p)\).

\begin{proposition}[Strict-saddle basins]
\label{prop:strict-saddle-basins}
Under \eqref{eq:lyaprange}, \(\mathcal B(\mathcal S)\) has zero
measure.
\end{proposition}

\begin{proof}
We apply
\cite[Assumption~2.1, Theorem~2.2 and Corollary~2.3]{CheriditoJentzenRossmannek2024}
to the lift of the reduced map. Let $\varpi:E\to\mathcal Z$ be the
covering projection with deck group
$2\pi\mathbb Z^n\cap E$, and define
\[
 \widehat F_h(y)=y+hK\sin y,\qquad y\in E,
\]
so that $\varpi\circ\widehat F_h=\bar F_h\circ\varpi$. Put
$\widehat{\mathcal S}=\varpi^{-1}(\mathcal S)$, a set of fixed points of
$\widehat F_h$.

The cited assumption consists of conditions at each point of an
arbitrary set of fixed points; neither isolation nor countability is
required, and the cited corollary handles uncountable sets by
second countability. Since $\widehat F_h$ is analytic, it is
differentiable at every point of $\widehat{\mathcal S}$, with
first-order Taylor remainder locally $\varepsilon$-Lipschitz for every
$\varepsilon>0$. At lifts of isolated critical points,
Proposition~\ref{prop:hyperbolicity} shows that the derivative is
diagonalisable over $\mathbb R$ with an eigenvalue of modulus greater
than one; at lifts of points of critical circles other than $p_\pm$,
Proposition~\ref{prop:circlemultipliers} gives the same conclusion, the
tangent direction contributing a multiplier equal to one. By the cited
theorem, every point whose forward orbit remains near such a fixed
point lies in a Lipschitz graph over the centre-stable subspace, which
has codimension at least one. The cited corollary, applied on
$E\simeq\mathbb R^{n-1}$, yields a null set $\widehat N\subset E$ such
that every orbit of $\widehat F_h$ converging to a point of
$\widehat{\mathcal S}$ eventually enters $\widehat N$.

Now let $\bar F_h^{\,k}(z_0)\to p\in\mathcal S$, choose a lift $y_0$ of
$z_0$, and put $y_k=\widehat F_h^{\,k}(y_0)$, a lift of
$\bar F_h^{\,k}(z_0)$. Then $\|y_{k+1}-y_k\|=h\|K\sin y_k\|\to0$ and
$\operatorname{dist}(y_k,\varpi^{-1}(p))\to0$. Since $\varpi^{-1}(p)$
is a translate of a lattice,
let $\lambda_k\in\varpi^{-1}(p)$ be a nearest lattice
point to $y_k$. Then
\[
 \|\lambda_{k+1}-\lambda_k\|
 \le \operatorname{dist}(y_{k+1},\varpi^{-1}(p))
    +\|y_{k+1}-y_k\|
    +\operatorname{dist}(y_k,\varpi^{-1}(p))
 \longrightarrow 0.
\]
The lattice has positive minimum separation, so $\lambda_k$ is
eventually constant. Hence $y_k$ converges to a single point of
$\widehat{\mathcal S}$, and therefore
$y_0\in\bigcup_{k\ge0}\widehat F_h^{-k}(\widehat N)$. As $\varpi$ is a
local diffeomorphism, $N:=\varpi(\widehat N)$ is null in $\mathcal Z$, and
\[
 \mathcal B(\mathcal S)\subset\bigcup_{k\ge0}\bar F_h^{-k}(N).
\]
Proposition~\ref{prop:null-preimages} makes every set in this union null.
\end{proof}

\FloatBarrier
\section{The completely degenerate intersections}
\label{sec:parabolic}

Throughout this section \(4\mid n\) and \(0<h<1/2\), the Lyapunov
interval for even rings. The involution \(z\mapsto-z\) commutes with
\(\bar F_h\) and exchanges \(p_+\) and \(p_-\), so it suffices to
treat \(p_+=\frac\pi2\ones\). In a local lift, write \(z=p_++y\) with
\(y\in E\) and set
\[
 G_h(y):=y+hK\cos y,\qquad
 W(y):=V(p_++y)=-\sum_{i=1}^n\sin y_i.
\]
With the quotient metric \(g\) of \eqref{eq:metric}, define
\[
 P_3(y):=\frac16\sum_{i=1}^ny_i^3,
 \qquad
 S_g:=\{u\in E:\|u\|_g=1\},
 \qquad
 \phi:=P_3|_{S_g}.
\]
For \(y\ne0\), write \(y=ru\) with \(r=\|y\|_g\) and \(u\in S_g\),
and \((r^+,u^+)\) for the polar coordinates of \(G_h(y)\). We denote by
\(\Delta_g\) the Laplacian of \(g\) and by
\(\operatorname{grad}_{S_g}\) the gradient on \(S_g\).

The argument runs as follows. At $p_+$ the derivative of $\bar F_h$ is
the identity, so linearisation gives no information, and we blow up the
origin. To leading order $G_h$ is a gradient step for the cubic $P_3$:
in polar coordinates the radius decreases at the rate
$3h\phi(u)r^2$, while the direction $u$ follows the gradient flow of
$\phi$ on $S_g$, slowed down by the factor $r$
(Lemma~\ref{lem:local-normal-form}). Along an orbit converging to
$p_+$, $\phi(u_k)$ tends to a critical value $\kappa\ge0$
(Lemma~\ref{lem:angular-limit}). If $\kappa>0$, the orbit approaches
$p_+$ at the polynomial rate $r_k\sim(3h\kappa k)^{-1}$; since $P_3$ is
harmonic in the metric $g$, the map preserves volume to first order, and
a change-of-variables argument shows that these orbits form a null set.
If $\kappa=0$, the direction $u_k$ converges to one of the finitely many
points of $\mathcal D_0$, which are the directions of the critical
circles at $p_+$ (Lemma~\ref{lem:zero-directions}). These are
non-degenerate saddles of $\phi$ lying on rays fixed by $G_h$, and the
slow-time dichotomy of Lemma~\ref{lem:slow-dichotomy} shows that the
angular deviation decays too fast for $r_k$ to tend to zero. The two
cases together give Proposition~\ref{prop:local-degenerate-basin}.

\begin{lemma}[Local normal form]
\label{lem:local-normal-form}
As \(y\to0\) in \(E\),
\begin{align}
 G_h(y)&=y-h\gradg P_3(y)+O(\|y\|_g^4),
 \label{eq:local-map-cubic}\\
 W(y)&=P_3(y)-\frac1{120}\sum_{i=1}^ny_i^5+O(\|y\|_g^7).
 \label{eq:local-potential-cubic}
\end{align}
Moreover,
\begin{equation}
\label{eq:harmonic-cubic}
 \Delta_gP_3=0,
\end{equation}
and, for \(y=ru\), the blown-up map satisfies
\begin{align}
 r^+&=r-3h\phi(u)r^2
 +\frac{h^2}{2}\|\operatorname{grad}_{S_g}\phi(u)\|_g^2r^3
 +O(r^4),
 \label{eq:radial-blowup}\\
 u^+&=u-hr\operatorname{grad}_{S_g}\phi(u)+O(r^2).
 \label{eq:angular-blowup}
\end{align}
\end{lemma}

\begin{proof}
Componentwise Taylor expansion gives
\[
 \cos y=\ones-\frac12y^{\circ2}+\frac1{24}y^{\circ4}
 +O(\|y\|^6).
\]
Since \(K\ones=0\),
\[
 G_h(y)=y-\frac h2K(y^{\circ2})
       +\frac h{24}K(y^{\circ4})+O(\|y\|^6).
\]
Since \(DP_3(y)v=\frac12\sum_i y_i^2v_i\) for \(v\in E\), we have
\(\gradg P_3(y)=\frac12K(y^{\circ2})\), which proves
\eqref{eq:local-map-cubic}.  Expanding \(-\sum_i\sin y_i\) and using
\(\sum_i y_i=0\) gives \eqref{eq:local-potential-cubic}.

The Hessian of \(P_3\) is the bilinear form
\(D^2P_3(y)[v,w]=\sum_i y_iv_iw_i\), and the inverse metric on \(E\) is
\(K_0\); hence
\[
 \Delta_gP_3(y)=\operatorname{tr}_E(K_0\operatorname{diag}y)
 =\operatorname{tr}(K\operatorname{diag}y)
 =2\sum_i y_i=0.
\]
Finally, put \(\xi(u)=\operatorname{grad}_{S_g}\phi(u)\). By homogeneity,
\(\gradg P_3(u)=3\phi(u)u+\xi(u)\), and the two terms are
\(g\)-orthogonal because \(\xi(u)\) is tangent to \(S_g\). Hence
\[
 \frac{\|G_h(ru)\|_g^2}{r^2}
 =1-6h\phi(u)r
   +h^2r^2\bigl(9\phi(u)^2+\|\xi(u)\|_g^2\bigr)+O(r^3).
\]
Taking the square root and normalising \(G_h(ru)\) proves
\eqref{eq:radial-blowup} and \eqref{eq:angular-blowup}.
The remainders are uniform on $S_g$.
\end{proof}

The zero-level angular critical set is
\[
 \mathcal D_0:=\{u\in S_g:
 \operatorname{grad}_{S_g}\phi(u)=0,\ \phi(u)=0\}.
\]

\begin{lemma}[Zero-level directions]
\label{lem:zero-directions}
Every \(u\in\mathcal D_0\) has coordinates of equal absolute value,
\(n/2\) of them positive. Each point of \(\mathcal D_0\) is a
non-degenerate saddle of \(\phi\), with \((n-2)/2\) positive and
\((n-2)/2\) negative Hessian eigenvalues.
\end{lemma}

\begin{proof}
At a critical point of the spherical restriction, homogeneity gives
\[
 \gradg P_3(u)=3\phi(u)u.
\]
For \(u\in\mathcal D_0\), this becomes \(K(u^{\circ2})=0\), so
\(u_i^2\) is constant because \(\ker K=\operatorname{span}\{\ones\}\),
and \(\sum_iu_i=0\) forces \(n/2\) coordinates of each sign:
\[
 u=a_\varepsilon\varepsilon,\qquad
 \varepsilon_i\in\{-1,1\},\qquad
 \sum_i\varepsilon_i=0,\qquad
 a_\varepsilon
 =\bigl(\varepsilon^TK_0^{-1}\varepsilon\bigr)^{-1/2}.
\]
Conversely, each such vector lies in \(\mathcal D_0\), since
\(K(u^{\circ2})=0\) and \(P_3(u)=0\). In particular \(\mathcal D_0\)
is finite.

The Hessian form of \(P_3\) is
\[
 D^2P_3(u)[v,v]=\sum_i u_iv_i^2.
\]
On \(E\) it has \((n-2)/2\) positive directions,
\((n-2)/2\) negative directions and one-dimensional kernel
\(\operatorname{span}\{u\}\), by the same block decomposition as in the
proof of Proposition~\ref{prop:circle-hessian}.  The tangent space
\(T_uS_g\) is a complement to that kernel, so the restricted Hessian is
non-degenerate with the stated inertia.
\end{proof}

For an orbit \(y_k=G_h^k(y_0)\) with \(y_k\to0\) and \(y_k\ne0\) for
all \(k\), write \(r_k=\|y_k\|_g\) and \(u_k=y_k/r_k\).

\begin{lemma}[Angular limit]
\label{lem:angular-limit}
For every such orbit, there is \(\kappa\ge0\) such that
\begin{equation}
\label{eq:angular-critical-limit}
 \phi(u_k)\to \kappa,
 \qquad
 \operatorname{dist}\bigl(u_k,
 \operatorname{Crit}(\phi)\cap\phi^{-1}(\kappa)\bigr)\to0.
\end{equation}
\end{lemma}

\begin{proof}
Strict descent and \(W(y_k)\to W(0)=0\) imply \(W(y_k)>0\)
for all \(k\).  By \eqref{eq:local-potential-cubic},
\[
 W(r_ku_k)=r_k^3\bigl(\phi(u_k)+O(r_k^2)\bigr),
\]
so every angular accumulation point has non-negative \(\phi\)-value.

By \eqref{eq:radial-blowup}, \(|r_{k+1}-r_k|\le Cr_k^2\); since
\(r_{k+1}/r_k\to1\), it follows that \(|1/r_{k+1}-1/r_k|\le C'\) for
large \(k\). Thus \(r_k\ge c_0/(k+1)\) eventually, with \(c_0>0\), and
\(\sum_k r_k=\infty\).

Put \(a_k=hr_k\) and introduce the slow time
\[
 \tau_k=\sum_{j<k}a_j.
\]
Then \(a_k\to0\), \(\sum_k a_k=\infty\), and
\eqref{eq:angular-blowup} takes the form
\[
 u_{k+1}=u_k-a_k\operatorname{grad}_{S_g}\phi(u_k)+O(a_k^2),
\]
with a uniform remainder.  Join consecutive points by the short
geodesic on \(S_g\), parametrised over
\([\tau_k,\tau_{k+1}]\), and denote the resulting interpolation by
\(\widehat u(\tau)\).  The vector field
\(-\operatorname{grad}_{S_g}\phi\) is smooth on the compact sphere.
For each fixed \(T>0\), the usual local Euler estimate followed by
a discrete Gronwall inequality gives
\[
 \sup_{0\le t\le T}
 d_g\bigl(\widehat u(\tau_k+t),
              \Phi_t(u_k)\bigr)
 \le C_T\sup_{j\ge k}a_j\longrightarrow0,
\]
where \(\Phi_t\) is the flow of
\(\dot u=-\operatorname{grad}_{S_g}\phi(u)\).
Indeed, the accumulated one-step error over any slow-time interval of
length at most \(T+\sup a_j\) is bounded by
\(C\sum a_j^2\le C_T\sup_{j\ge k}a_j\).
Thus \(\widehat u\) is a precompact asymptotic pseudotrajectory of
\(\Phi\).

By the limit-set theorem
\cite[Theorem~5.7(i)]{Benaim1999}, its limit set is internally
chain transitive.  Since \(\phi\) decreases strictly along
non-stationary trajectories of \(\Phi\) and, by Sard's theorem, its set
of critical values has empty interior,
\cite[Proposition~6.4]{Benaim1999} shows that the limit set lies in
\(\operatorname{Crit}(\phi)\) and that \(\phi\) is constant there.
The sequence \((u_k)\) has the same limit set as its interpolation,
because \(d_g(u_{k+1},u_k)=O(a_k)\to0\), and the constant value \(\kappa\)
is non-negative by the first paragraph.
\end{proof}

For a sufficiently small ball \(U\subset E\) centred at the origin, let
\[
 \mathcal B^+_{\rm loc}:=
 \{y\in U:G_h^k(y)\in U\ \text{for all }k\ge0,
                   \ G_h^k(y)\to0\}.
\]

\begin{proposition}[Local basin of the degenerate intersection]
\label{prop:local-degenerate-basin}
The set \(\mathcal B^+_{\rm loc}\) has zero measure.
\end{proposition}

\begin{proof}
Choose a convex ball \(U\) small enough that
\(\sup_U\|DG_h-I\|<1/2\).  Integrating \(DG_h-I\) along the
segment joining any two points of \(U\) shows that
\[
 \|G_h(x)-G_h(y)\|\ge\tfrac12\|x-y\|,
 \qquad x,y\in U.
\]
Thus \(G_h\) is injective on \(U\) with positive Jacobian, and no
nonzero orbit remaining in \(U\) reaches the origin in finite time. For
such an orbit in \(\mathcal B^+_{\rm loc}\), let \(\kappa\) be given by
Lemma~\ref{lem:angular-limit}.

Suppose first that \(\kappa>0\).  By \eqref{eq:radial-blowup},
\[
 r_{k+1}=r_k-3h\kappa\,r_k^2+o(r_k^2).
\]
Hence \(1/r_{k+1}-1/r_k\to3h\kappa\), so \(r_k\sim(3h\kappa\,k)^{-1}\) and
\(\sum_kr_k^2<\infty\). Write \(DG_h(y)=I+M(y)\) with
\(M(y)=-hK\operatorname{diag}(\sin y)|_E\). Then
\(\|M(y)\|=O(\|y\|_g)\) and
\[
 \operatorname{tr}_E M(y)
 =-h\operatorname{tr}(K\operatorname{diag}(\sin y))
 =-2h\sum_i\sin y_i
 =O(\|y\|_g^3),
\]
the cancellation of the linear term being the one behind
\eqref{eq:harmonic-cubic}. Since
\(\log\det(I+M)=\operatorname{tr}M+O(\|M\|^2)\), uniformly for
\(y\in U\),
\begin{equation*}
 \log\det DG_h(y)=O(\|y\|_g^2).
\end{equation*}
Hence the products
\(\mathcal J_k(y_0)=\det DG_h^k(y_0)\) converge to a number
\(\mathcal J_\infty(y_0)\in(0,\infty)\) along every orbit of this type.

Let \(\mathcal E_+\subset\mathcal B^+_{\rm loc}\) be the set of
initial points with \(\kappa>0\). If it had positive measure, some
\(\mathcal E_m=\{y\in\mathcal E_+:\mathcal J_\infty(y)\ge1/m\}\) would have positive
measure.
By Egorov's theorem, applied to \(\mathcal J_k\) and \(G_h^k\) on
\(\mathcal E_m\), both converge uniformly on a measurable subset
\(\mathcal E'\subset\mathcal E_m\) of positive measure; in particular
\(\mathcal J_k\ge1/(2m)\) on \(\mathcal E'\) for large \(k\). Since orbits stay
in \(U\), \(G_h^k\) is injective on \(\mathcal E'\), and the change of
variables formula gives
\[
 \operatorname{Leb}(G_h^k(\mathcal E'))
 =\int_{\mathcal E'} \mathcal J_k(y)\,dy
 \ge\frac{\operatorname{Leb}(\mathcal E')}{2m}>0
\]
for all large \(k\), whereas \(G_h^k(\mathcal E')\) lies in balls about
the origin with radii tending to zero. Hence \(\mathcal E_+\) is null.

It remains to exclude \(\kappa=0\). Then
\(\operatorname{dist}(u_k,\mathcal D_0)\to0\) by
Lemma~\ref{lem:angular-limit}, \(\mathcal D_0\) is finite by
Lemma~\ref{lem:zero-directions}, and \(d_g(u_{k+1},u_k)\to0\) by
\eqref{eq:angular-blowup}. Since the points of \(\mathcal D_0\) are at
positive distance from one another, the vanishing increments prevent
\((u_k)\) from switching between neighbourhoods of different points;
hence \(u_k\to u_0\) for a single \(u_0\in\mathcal D_0\).
Write \(u_{0,i}=a\varepsilon_i\), with \(a>0\) and
\(\varepsilon_i\in\{-1,1\}\).  The ray through \(u_0\) is
pointwise fixed by \(G_h\), since
\[
 G_h(ru_0)=ru_0+hK\bigl(\cos(ar)\ones\bigr)=ru_0.
\]

Choose \(g\)-normal angular coordinates \(\eta\) near \(u_0\),
so that the Hessian matrix
\(H_\phi=D^2_{S_g}\phi(u_0)\) is symmetric and nonsingular by
Lemma~\ref{lem:zero-directions}.  Denote the radial and angular
components of the polar map by \(\rho(r,\eta)\) and
\(\mathcal U(r,\eta)\).  They extend smoothly to \(r=0\), where
\(\rho(0,\eta)=0\) and \(\mathcal U(0,\eta)=\eta\).
The angular part fixes \(\eta=0\) for every small
\(r\), and \eqref{eq:angular-blowup} gives
\[
 \mathcal U(r,0)=0,\qquad
 D_\eta\mathcal U(r,0)=I-hrH_\phi+O(r^2),
\]
with \(D_\eta^2\mathcal U(r,\eta)=O(r)\), uniformly near
\((0,0)\).  Taylor's theorem therefore yields
\begin{equation*}
 \eta_{k+1}=\bigl[I-hr_kH_\phi+R_k\bigr]\eta_k,
 \qquad
 \|R_k\|\le C\bigl(r_k\|\eta_k\|+r_k^2\bigr)=o(r_k).
\end{equation*}
Apply Lemma~\ref{lem:slow-dichotomy} with \(a_k=hr_k\)
and \(H_*=H_\phi\).  Because \(\eta_k\to0\) and
\(\sum_k r_k=\infty\), there are \(C,\gamma>0\) such that
\begin{equation}
\label{eq:angular-exponential}
 \|\eta_k\|\le Ce^{-\gamma\tau_k},
 \qquad \tau_k=h\sum_{j<k}r_j.
\end{equation}

The fixed ray gives \(\rho(r,0)=r\); moreover
\(D_\eta\rho(r,0)=0\) for small \(r\). Indeed, let
\(v\in T_{u_0}S_g\). Since \(\sin(ru_0)=\sin(ar)\varepsilon\) and
\(G_h(ru_0)=ru_0\ne0\), the \(g\)-norm is differentiable at
\(G_h(ru_0)\), and differentiating \(\|G_h(ru)\|_g\) at \(u=u_0\)
gives
\[
 D_\eta\rho(r,0)v
 =r\langle u_0,v\rangle_g
  -hr\sin(ar)
   \langle u_0,K\operatorname{diag}(\varepsilon)v\rangle_g=0.
\]
The first inner product vanishes because \(v\) is tangent to \(S_g\),
and the second equals
\[
 \langle u_0,K\operatorname{diag}(\varepsilon)v\rangle_g
 =u_0^T\operatorname{diag}(\varepsilon)v
 =a\sum_i v_i=0,
\]
as \(v\in E\).  Writing \(G_h(ru)=ru+r^2Q(r,u)\) with \(Q\) smooth for small
\(r\ge0\), we get
\[
 \rho(r,\eta)=r\|u(\eta)+rQ(r,u(\eta))\|_g
             =r+r^2q(r,\eta),
\]
where \(q\) is smooth with uniformly bounded angular second
derivatives.  The two identities above give \(q(r,0)=0\) and \(D_\eta q(r,0)=0\),
so Taylor's theorem yields
\begin{equation}
\label{eq:radial-quadratic-angle}
 |\rho(r,\eta)-r|\le Cr^2\|\eta\|^2.
\end{equation}
Since $r_{k+1}=\rho(r_k,\eta_k)$, \eqref{eq:radial-quadratic-angle}
gives
\[
 \Bigl|\frac{r_{k+1}}{r_k}-1\Bigr|\le Cr_k\|\eta_k\|^2\longrightarrow0.
\]
Using $|\log(1+t)|\le2|t|$ for $|t|\le\frac12$ and
\eqref{eq:angular-exponential}, we obtain, for large $k$,
\[
 |\log r_{k+1}-\log r_k|\le 2Cr_k\|\eta_k\|^2
 \le C'r_ke^{-2\gamma\tau_k}.
\]
Since $\tau_{k+1}-\tau_k=hr_k\to0$ and $1-e^{-x}\ge x/2$ for
$0\le x\le1$, for large $k$
\[
 r_ke^{-2\gamma\tau_k}
 \le\frac{1}{h\gamma}
     \bigl(e^{-2\gamma\tau_k}-e^{-2\gamma\tau_{k+1}}\bigr),
\]
and the right-hand side telescopes. Hence
$\sum_k|\log r_{k+1}-\log r_k|<\infty$, so $\log r_k$ converges to a
finite limit, contradicting $r_k\to0$.

The mechanism is the following. On the invariant ray the radius is
preserved exactly; off the ray, the radial change is quadratic in the
angular deviation $\eta$; and $\eta$ decays exponentially in slow time.
At the zero level the cubic term provides no radial contraction, so
nothing can drive $r_k$ to zero. Hence no nonzero orbit in
$\mathcal B^+_{\rm loc}$ has $\kappa=0$, and the origin itself has zero
measure.
\end{proof}

\begin{proposition}[Completely degenerate basins]
\label{prop:degenerate-basins}
If \(4\mid n\), then under \eqref{eq:lyaprange} the basins
\(\mathcal B(p_+)\) and \(\mathcal B(p_-)\) have zero measure.
\end{proposition}

\begin{proof}
By Proposition~\ref{prop:local-degenerate-basin} and the involution
\(z\mapsto-z\), the local basins of \(p_\pm\) are null. An orbit
converging to \(p_\pm\) enters the local basin after finitely many
iterates, so Proposition~\ref{prop:null-preimages} gives the global
statement.
\end{proof}

\begin{theorem}[Almost-global convergence to uniform twists]
\label{thm:almost-global}
For every \(n\ge3\) and every \(h\) in \eqref{eq:lyaprange}, almost
every orbit of \(\bar F_h\) converges to a uniform twist that is a local
minimum of \(V\).
\end{theorem}

\begin{proof}
Every orbit converges to a critical point
(Proposition~\ref{prop:global}); the basins of the non-minimal ones are
null (Proposition~\ref{prop:strict-saddle-basins} and, when
\(4\mid n\), Proposition~\ref{prop:degenerate-basins}); and the local minima are the uniform
twists with \(c_q<0\) (Theorem~\ref{thm:minima}).
\end{proof}

The theorem asserts that the basins of the locally minimising twists
together have full measure; it does not single out one twist. Since
\(z\mapsto -z\) commutes with \(\bar F_h\), preserves measure and
exchanges the twists \(q\) and \(n-q\), their basins have equal
measure.

For identical oscillators on unweighted simple graphs, the
harmonicity calculation \eqref{eq:harmonic-cubic} extends to the
completely degenerate equilibria of \cite{Sclosa2024}: the edge
Laplacian has diagonal entries two, and the linear term of the potential
vanishes at equilibrium. The cycle structure is used only at the zero
level, through the non-degenerate angular saddles and their pointwise
fixed rays; extending the blow-up argument to other graphs requires an
analysis of the zero-level directions there.

\FloatBarrier
\section{Small rings: \texorpdfstring{$C_3$ to $C_8$}{C3 to C8}}
\label{sec:examples}

Throughout this section \(h\) lies in the Lyapunov interval
\eqref{eq:lyaprange}. The examples show the first additional twisted
minima and, for \(C_4\) and \(C_8\), the critical circles.

\subsection{The triangle and square: phase opposition and critical circles}
\label{subsec:C4}

The triangle $C_3$ is also the complete graph $K_3$, so nearest-neighbour
and all-to-all coupling coincide.  Its six critical points have
Morse index counts $(2,3,1)$ for indices $0,1,2$.  The two minima have
uniform edge differences $2\pi/3$ and $4\pi/3$.  With $\theta_3=0$, their
relative phases are $(\theta_1,\theta_2)=(2\pi/3,4\pi/3)$ and
$(4\pi/3,2\pi/3)$, as in the original three-oscillator analysis
\cite{DAnielloOliveira2024}. The involution $z\mapsto-z$ exchanges their
basins, which by Theorem~\ref{thm:almost-global} have $\mu$-measure
$1/2$ each.
The configurations are shown in Figure~\ref{fig:c3-final-states}.

\begin{figure}[!htbp]
\centering
\includegraphics[width=0.66\textwidth]{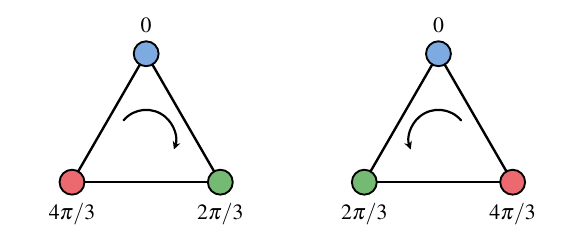}
\caption{The two uniform twisted minima of $C_3$, with edge differences
$2\pi/3$ and $4\pi/3$. The top oscillator has phase zero, and the
arrows indicate the opposite phase orientations. Here and in
Figures~\ref{fig:C4}--\ref{fig:C8}, the labels are vertex phases
$\theta_i$; the edge differences are their successive differences
around the ring.}
\label{fig:c3-final-states}
\end{figure}

The square $C_4$ has two isolated critical points: the maximum
$z=0$ with $V=4$, and the unique minimum $z=\pi\ones$ with $V=-4$.
Three critical circles, on which $V=0$, meet at $p_\pm$. Their
geometry is described in \cite[Section~5.2]{Sclosa2023}; see also
\cite[Example~5.4]{ChenKorchevskaiaLindberg2026}. The circles and the
minimum are shown in Figures~\ref{fig:C4criticalcircles}
and~\ref{fig:C4}.

\begin{figure}[!htbp]
\centering
\includegraphics[width=0.57\textwidth]{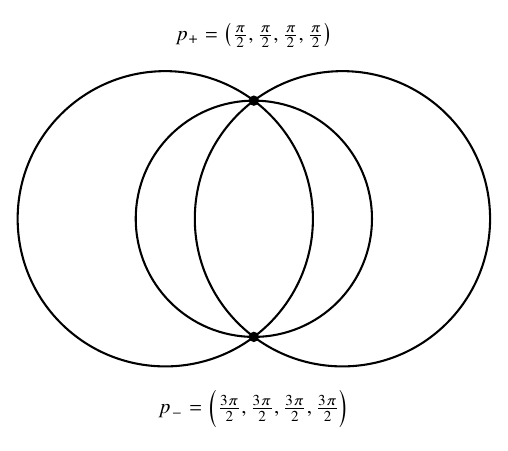}
\caption{The three critical circles of $C_4$, shown schematically.
All three intersect at the two completely degenerate configurations
$p_+$ and $p_-$, shown in edge coordinates; see
\cite[Section~5.2]{Sclosa2023}.}
\label{fig:C4criticalcircles}
\end{figure}

\begin{figure}[!htbp]
\centering
\includegraphics[width=0.43\textwidth]{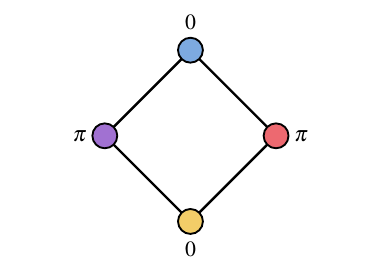}
\caption{The unique minimum of $C_4$: phase opposition on each edge,
with alternating vertex phases $0,\pi,0,\pi$.}
\label{fig:C4}
\end{figure}

\FloatBarrier

\subsection{The pentagon and hexagon: the first additional minima}
\label{subsec:C6}

Both $C_5$ and $C_6$ have finite, non-degenerate critical sets, with
Morse index counts, starting at index zero,
\begin{align*}
 C_5:&\quad (2,5,10,10,3),\qquad &&\#\Crit(V)=30,\\
 C_6:&\quad (3,12,15,15,12,3), &&\#\Crit(V)=60.
\end{align*}
Table~\ref{tab:c5} gives the critical values for $C_5$.

\begin{table}[htbp]
\centering
\caption{Critical configurations of $C_5$.  There are $k$ edges with
difference $\pi-\alpha$ and $5-k$ with difference $\alpha$, all modulo
$2\pi$.  Each count includes all choices of the $k$ positions and both
signs where indicated.}
\label{tab:c5}
\begin{tabular}{@{}ccccc@{}}
\toprule
$k$ & $\alpha$ & Number & $V$ & Morse index\\
\midrule
$0$ & $\pm4\pi/5$ & $2$ & $-\frac54(1+\sqrt5)$ & $0$\\
$1$ & $\pi$ & $5$ & $-3$ & $1$\\
$2$ & $0$ & $10$ & $1$ & $2$\\
$1$ & $\pm\pi/3$ & $10$ & $3/2$ & $3$\\
$0$ & $\pm2\pi/5$ & $2$ & $\frac54(\sqrt5-1)$ & $4$\\
$0$ & $0$ & $1$ & $5$ & $4$\\
\bottomrule
\end{tabular}
\end{table}

The two minima of $C_5$ have uniform edge differences $4\pi/5$
and $6\pi/5$, and, as for $C_3$, their basins have measure $1/2$ each
(Figure~\ref{fig:c5-final-states}).

\begin{figure}[!htbp]
\centering
\includegraphics[width=0.66\textwidth]{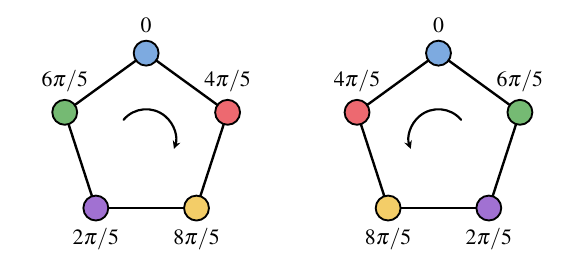}
\caption{The two uniform twisted minima of $C_5$, with edge differences
$4\pi/5$ and $6\pi/5$, and opposite phase orientations.}
\label{fig:c5-final-states}
\end{figure}
\FloatBarrier

The hexagon is the smallest even ring with additional local minima.  Its
three minima have uniform edge differences
\begin{equation*}
 z_i=\frac{2\pi}{3},\qquad z_i=\pi,\qquad z_i=\frac{4\pi}{3}
 \quad\text{for all }i.
\end{equation*}
Phase opposition has potential $-6$; the other two minima have
potential $-3$ and basins of equal measure (Figure~\ref{fig:C6}).

\begin{figure}[!htbp]
\centering
\includegraphics[width=0.94\textwidth]{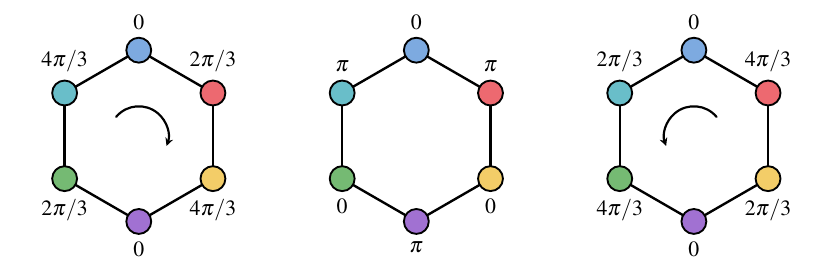}
\caption{The three minima of $C_6$. The left and right configurations
have uniform edge differences $2\pi/3$ and $4\pi/3$, respectively,
and potential $-3$. The central configuration is the unique absolute
minimum, with phase opposition and potential $-6$.}
\label{fig:C6}
\end{figure}

\FloatBarrier
\subsection{Seven and eight oscillators: additional minima and degeneracy}
\label{subsec:C8}

The ring $C_7$ is the first odd ring with minima above the absolute
minimum value.  Among its $140$ critical points, four are minima, with
uniform edge differences
\[
 \frac{4\pi}{7},\quad\frac{6\pi}{7},\quad
 \frac{8\pi}{7},\quad\frac{10\pi}{7}.
\]
The middle pair are the absolute minima and the outer pair the
additional ones (Figure~\ref{fig:c7-final-states}).

\begin{figure}[htbp]
\centering
\includegraphics[width=0.78\textwidth]{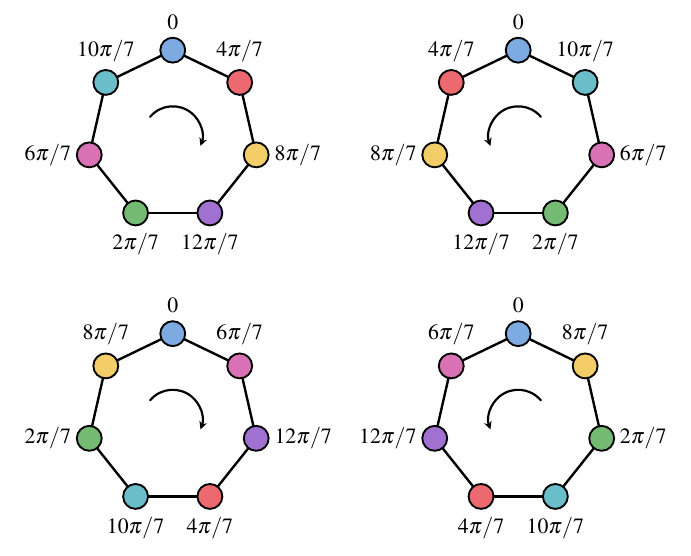}
\caption{The four minima of $C_7$. The top row shows the additional
minima, with uniform edge differences $4\pi/7$ and $10\pi/7$
from left to right. The bottom row shows the absolute minima,
with differences $6\pi/7$ and $8\pi/7$ from left to right.}
\label{fig:c7-final-states}
\end{figure}

The ring $C_8$ has $94$ isolated critical points and $35$ critical
circles meeting at $p_\pm$. Its three minima have uniform edge differences
\begin{equation*}
 z_i=\frac{3\pi}{4},\qquad z_i=\pi,\qquad z_i=\frac{5\pi}{4}
 \quad\text{for all }i,
\end{equation*}
with potential values $-4\sqrt2$, $-8$ and $-4\sqrt2$
(Figure~\ref{fig:C8}).

\begin{figure}[htbp]
\centering
\includegraphics[width=0.72\textwidth]{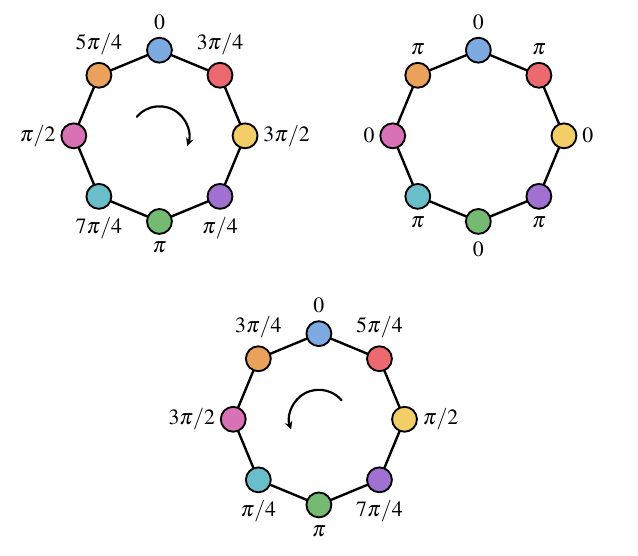}
\caption{The three minima of $C_8$. The configurations at the top left
and bottom have uniform edge differences $3\pi/4$ and $5\pi/4$,
respectively. The configuration at the top right is the
phase-opposition state and the unique absolute minimum.}
\label{fig:C8}
\end{figure}

Table~\ref{tab:growth} gives the counts for larger rings.

\begin{table}[htbp]
\centering
\caption{Critical sets and sinks for $0<h<2/\Lambda_{n}$. Isolated
points are non-degenerate; the two degenerate intersections of the
critical circles are not counted. Additional sinks are the minima other
than the absolute minima.}
\label{tab:growth}
\vspace{0.55em}
\begin{tabular}{@{}lrrrr@{}}
\toprule
Ring & Isolated points & Critical circles & Sinks & Additional sinks\\
\midrule
$C_3$    & $6$     & $0$   & $2$ & $0$\\
$C_4$    & $2$     & $3$   & $1$ & $0$\\
$C_5$    & $30$    & $0$   & $2$ & $0$\\
$C_6$    & $60$    & $0$   & $3$ & $2$\\
$C_7$    & $140$   & $0$   & $4$ & $2$\\
$C_8$    & $94$    & $35$  & $3$ & $2$\\
$C_9$    & $630$   & $0$   & $4$ & $2$\\
$C_{10}$ & $1260$  & $0$   & $5$ & $4$\\
$C_{11}$ & $2772$  & $0$   & $6$ & $4$\\
$C_{12}$ & $2372$  & $462$ & $5$ & $4$\\
$C_{13}$ & $12012$ & $0$   & $6$ & $4$\\
$C_{14}$ & $24024$ & $0$   & $7$ & $6$\\
\bottomrule
\end{tabular}
\end{table}
\FloatBarrier

\section{Conclusions}
\label{sec:discussion}

For every ring size and every interaction strength in the Lyapunov
interval, almost every orbit of the discrete Huygens map converges to a
uniform twisted state. The limit is not unique: for \(n\ge6\) there are
attracting twists other than the absolute minima of the potential, so
the initial phases determine which twist is reached, and convergence
does not select the state of minimal potential. When \(4\mid n\), the
critical set contains circles of critical points, but both the circles
and their completely degenerate intersections have basins of zero
measure. Where an oscillator ring must reach a prescribed phase
pattern, this multistability makes the choice of initial phases part of
the design.

Two questions remain open. On odd rings, the absolute minima remain
sinks up to the threshold \eqref{eq:odd-stability-gap}, beyond the
Lyapunov interval, and it is not known whether almost-global
convergence extends to this range. For larger interaction strengths the
uniform twists lose stability, and the resulting bifurcations are not
treated here.

\appendix
\section{A slow-time dichotomy}
\label{app:slow-dichotomy}

All norms and matrices in this appendix are Euclidean.

Let \(H_*\) be a symmetric nonsingular matrix.  Let \((a_k)\) be a
positive sequence with \(a_k\to0\) and \(\sum_ka_k=\infty\), and let
\((R_k)\) satisfy \(\|R_k\|=o(a_k)\).  Consider
\begin{equation}
\label{eq:slow-dichotomy-system}
 x_{k+1}=(I-a_kH_*+R_k)x_k,
 \qquad
 \sigma_k:=\sum_{j<k}a_j.
\end{equation}

\begin{lemma}[Slow-time dichotomy]
\label{lem:slow-dichotomy}
Every solution of \eqref{eq:slow-dichotomy-system} that converges to zero
satisfies
\begin{equation}
\label{eq:slow-dichotomy-decay}
 \|x_k\|\le Ce^{-\gamma\sigma_k}
\end{equation}
for some \(C,\gamma>0\).
\end{lemma}

\begin{proof}
After an orthogonal change of coordinates, write
\(
H_*=\operatorname{diag}(H_s,-H_u),
\)
where \(H_s\) and \(H_u\) are positive definite whenever their
blocks are non-empty.  Let \(P_s\) and \(P_u\) denote the corresponding
projections, and omit the respective terms if a block is absent.
Choose $\lambda>0$ below the least eigenvalue among the non-empty
positive blocks. For $j\le k$ and $j\ge k$, respectively, set
\[
 \Phi_s(k,j)=\prod_{\ell=j}^{k-1}(I-a_\ell H_s),\qquad
 \Phi_u(k,j)=\prod_{\ell=k}^{j-1}(I+a_\ell H_u)^{-1},
\]
with empty products equal to $I$ and products taken in increasing
index order. For sufficiently large indices these operators satisfy
\[
 \|\Phi_s(k,j)\|\le C e^{-\lambda(\sigma_k-\sigma_j)},
 \qquad j\le k,
\]
and the inverse products on the unstable block satisfy
\[
 \|\Phi_u(k,j)\|\le C e^{-\lambda(\sigma_j-\sigma_k)},
 \qquad j\ge k.
\]

Let \(x_k\to0\).  Variation of constants gives, for sufficiently large
\(k_0\),
\begin{align*}
 P_sx_k
 &=\Phi_s(k,k_0)P_sx_{k_0}
   +\sum_{j=k_0}^{k-1}\Phi_s(k,j+1)P_sR_jx_j,\\
 P_ux_k
 &=-\sum_{j=k}^{\infty}\Phi_u(k,j+1)P_uR_jx_j.
\end{align*}
The second identity follows by writing the formula up to a finite
terminal index and letting that index tend to infinity, using
\(P_ux_k\to0\). Thus every convergent solution is a fixed point of the
associated Green operator on bounded sequences.

Fix \(0<\gamma<\lambda\), and increase \(k_0\) so that
\(\|R_j\|\le\varepsilon a_j\) for \(j\ge k_0\), with
\(\varepsilon>0\) arbitrarily small.  For every \(\beta>0\),
writing \(A_0=\sup_{j\ge k_0}a_j<\infty\), comparison with the
integral of an exponential over each
\([\sigma_j,\sigma_{j+1}]\) gives the two bounds
\begin{align*}
 \sup_{k\ge k_0}\sum_{j=k_0}^{k-1}
  a_j e^{-\beta(\sigma_k-\sigma_{j+1})}
 &\le \frac{e^{\beta A_0}}{\beta},\\
 \sup_{k\ge k_0}\sum_{j=k}^{\infty}
  a_j e^{-\beta(\sigma_j-\sigma_k)}
 &\le \frac{e^{\beta A_0}}{\beta}.
\end{align*}
With $\beta=\lambda$, these bounds show that the Green operator has
norm $O(\varepsilon)$ on $\ell^\infty$. After multiplication by
\(e^{\gamma\sigma_k}\) the stable and unstable exponents become
\(\lambda-\gamma>0\) and \(\lambda+\gamma>0\), so the same holds on
\[
 X_\gamma=
 \{x:\sup_{k\ge k_0}e^{\gamma\sigma_k}\|x_k\|<\infty\}.
\]
Choosing \(\varepsilon\) sufficiently small makes both operators
strict contractions.
Since the inhomogeneous term $\Phi_s(k,k_0)P_sx_{k_0}$ belongs to
$X_\gamma$, the Green equation has a unique fixed point in $X_\gamma$.
As $X_\gamma\subset\ell^\infty$, this is also the unique bounded fixed
point, which is the given solution. Thus
\(
\sup_{k\ge k_0}e^{\gamma\sigma_k}\|x_k\|<\infty.
\)
This is \eqref{eq:slow-dichotomy-decay}; enlarging the constant absorbs
the finite initial segment.
\end{proof}

\section*{Declarations}

\noindent\textbf{Data availability.} No datasets were generated or analysed during the preparation of this article. All proofs and calculations are presented in the paper.

\medskip
\noindent\textbf{Author contributions.} Both authors contributed equally to all stages of the work. Both authors reviewed and approved the final manuscript.

\medskip
\noindent\textbf{Funding.} Jorge Buescu acknowledges partial support from Funda\c{c}\~ao para a Ci\^encia e a Tecnologia (FCT), Portugal, through project UID/04561/2025 (DOI: 10.54499/\allowbreak UID/04561/2025). Henrique M. Oliveira acknowledges partial support from FCT, Portugal, through CAMGSD -- Centro de An\'alise Matem\'atica, Geometria e Sistemas Din\^amicos, under projects UID/04459/2025 (DOI: 10.54499/\allowbreak UID/04459/2025) and UID/PRR/04459/2025 (DOI: 10.54499/\allowbreak UID/PRR/04459/2025), financed by national funds and by the European Union through the Recovery and Resilience Plan (PRR) -- NextGenerationEU.

\medskip
\noindent\textbf{Competing interests.} The authors declare that they have no competing financial or non-financial interests that could have influenced the work reported in this manuscript.

\medskip
\noindent\textbf{Use of artificial intelligence.}
Generative AI was used for language editing, editorial assistance,
and during revision to assist in checking the mathematical
exposition and arguments. All mathematical arguments, proofs, results
and final wording were defined by the authors.

\begingroup
\small

\bibliographystyle{plain}
\bibliography{rings}
\endgroup
\end{document}